\documentclass[preprint]{imsart}

\RequirePackage[OT1]{fontenc}
\RequirePackage{pdfsync}
\RequirePackage{amsfonts}
\RequirePackage[numbers]{natbib}
\RequirePackage[colorlinks,citecolor=blue,urlcolor=blue]{hyperref}
\RequirePackage[centertags, leqno, sumlimits, nointlimits, namelimits]{mathtools}
\RequirePackage[thmmarks,thref,amsthm,amsmath]{ntheorem}
\RequirePackage{paralist}
\RequirePackage[english, noabbrev]{cleveref}

\startlocaldefs
\numberwithin{equation}{section}
\theoremstyle{plain}
\newtheorem{theorem}{Theorem}[section]
\newtheorem{lemma}[theorem]{Lemma}

\theoremstyle{remark}
\newtheorem{remark}[theorem]{Remark}

\newtheorem{assumption}[theorem]{Assumption}

\endlocaldefs

\def\N{{ \mathbb{N} }} 
\def\R{{ \mathbb{R} }} 
\def\P{{ \text{pr} }} 
\def\E{{ E }}		

\usepackage{multirow,color}

\DeclareMathOperator{\Exp}{Exp}
\begin{document}

\begin{frontmatter}
\title{More powerful logrank permutation tests for two-sample survival data}
\runtitle{Logrank permutation tests}

\begin{aug}
\author{\fnms{Marc} \snm{Ditzhaus}\ead[label=e1]{marc.ditzhaus@hhu.de}}
and 
\author{\fnms{Sarah} \snm{Friedrich}\ead[label=e2]{sarah.friedrich@alumni.uni-ulm.de}}


\runauthor{M. Ditzhaus and S. Friedrich}

\affiliation{Ulm University \thanksmark{m1}}

\address{Address of the authors\\
 Ulm University \\
 Institute of Statistics\\
  Helmholtzstr. 20\\
   89081 Ulm, Germany\\
\printead{e1}\\
\printead{e2}}

\end{aug}

\begin{abstract}
	Weighted logrank tests are a popular tool for analyzing right censored survival data from two independent samples. Each of these tests is optimal against a certain hazard alternative, for example the classical logrank test for proportional hazards. But which weight function should be used in practical applications? We address this question by a flexible combination idea leading to a testing procedure with broader power. Beside the test's asymptotic exactness and consistency its power behaviour under local alternatives is derived. All theoretical properties can be transferred to a permutation version of the test, which is even finitely exact under exchangeability and showed a better finite sample performance in our simulation study. The procedure is illustrated in a real data example.
\end{abstract}


\end{frontmatter}

\section{Introduction}\label{sec:intro}
Deciding whether there is a difference between two treatments is only one example for the variety of two-sample problems. Within the right censoring survival set-up the classical logrank test, first proposed by \citet{Mantel1966} and \citet{PetoPeto1972}, is very popular in practice. It is well known that the logrank test is optimal for proportional hazard alternatives but may  lead to wrong decisions when the relationship of the hazards is time-dependent. Adding a weight function we obtain optimal tests for other kinds of alternatives. These so-called weighted logrank tests are well studied in the literature, see \citet{ABGK}, \citet{BagonaviciusETAL2010}, \citet{FlemingHarrington}, \citet{HarringtonFleming1982}, \citet{Gill_1980}, \citet{KleinMoeschberger1997}, \citet{TaroneWare1977}. However, no weighted logrank test is a so-called omnibus test, i.e., a consistent test for all alternatives. Depending on the pre-chosen weight the corresponding logrank test is consistent for specific alternatives, details can be found in Section \ref{sec:set-up}. This is in line with the result of \citet{Janssen2000} that any test has only reasonable power for a finite dimensional subspace of the nonparametric two-sample alternative. A lot of effort was made to obtain tests with a good performance for a huge class of alternatives. \citet{FlemingETAL1987} suggested a supremum version of the logrank test with the purpose of power robustification. The funnel test of \citet{Ehm1995} had the same aim, to loose a little power for some alternatives and to gain a substantial power amount for other alternatives in reverse. \citet{LaiYing1991} proposed to estimate the weight function. Since they use kernel estimators a great amount of data is needed for a suitable performance and, hence, it is not usable for various applications. Adaptive weights were discussed by \citet{YangETAL2005} and \citet{YangPrentice2010}.  \citet{JonesCrowley1989,JonesCrowley1990} generalized many previous tests to a huge class of nonparametric single-covariate tests. Several researchers followed the idea to combine different weighted logrank tests. For instance,   \citet{Bajorski1992}, \citet{Tarone1981} and \citet{GaresETAL2015} took the maximum. \citet{BathkeETAL2009} considered the censored empirical likelihood with constraints corresponding to different weights. The supremum of function-indexed weighted logrank tests was studied by \citet{KosorokLin1999}. 

Finally, we like to focus on the paper of  \citet{BrendelETAL2014}, which motivated the present paper. Adapting the concept of broader power functions by \citet{BehnenNeuhaus1983,BehnenNeuhaus1989} to the right-censored survival set-up, they first choose a vector of weighted logrank statistics. Roughly speaking, this vector is then adaptively projected onto a space corresponding to the closed hazard alternative. In this way they ensure asymptotic optimality against the given alternatives of interest. A permutation version of their test solves the problem of the test statistic's unknown limit distribution. While their procedure is theoretically optimal (in some sense), it has the following disadvantages, which may explain why the method is not used in practice: 1. Due to the projection terminology the paper is quite hard to read and to understand. 2. Their permutation approach is computationally very expensive and time consuming. 3. Their method is not implemented in some common statistical software. In this paper we present a solution for all these points. 1. We only use the typical survival notation and our statistic is a simple quadratic form. 2. We explain how to appropriately choose the weights for the logrank statistic such that the asymptotic results are not affected but the corresponding permutation test becomes far more computationally effective. 3. Our novel method is implemented in an R package called \textit{mdir.logrank}, which is available on CRAN soon, and is very easy to use as illustrated in Section \ref{sec:real_data} by discussing a real data example. A simulation study promises a good finite sample performance of our permutation test under the null and a good power behaviour under various alternatives.

\section{Two-sample survival set-up}\label{sec:set-up}
We consider the standard two-sample survival set-up given by survival times $T_{j,i}\sim F_j$ and censoring times $C_{j,i}\sim G_{j}$ $(j=1,2;\,i=1,\ldots,n_j)$ with continuous distribution functions $F_j,G_j$ on the positive line. As usual, all random variables $T_{1,1},C_{1,1},\ldots,T_{2,n_2},C_{2,n_2}$ are assumed to be independent. Let $n=n_1+n_2$ be the pooled sample size, which is supposed to go to infinity in our asymptotic consideration. All limits $\to$ are meant as $n\to\infty$ if not stated otherwise. We are interested in the survival times' distributions $F_1,F_2$, but only the possibly censored survival times $X_{j,i}=\min(T_{j,i},C_{j,i})$ and their censoring status $\delta_{j,i}=\mathbf{1}\{X_{j,i}=T_{j,i}\}$ $(j=1,2;\,i=1,\ldots,n_j)$ are observable.

Throughout, we adopt the counting process notation of \citet{ABGK}. Let $N_{j,i}(t)=\mathbf{1}\{X_{j,i}\leq t,\,\delta_{j,i}=1\}$ and $Y_{j,i}(t)=\mathbf{1}\{X_{j,i}\geq t\}$ $(t\geq 0)$. Then $N_j(t)=\sum_{i=1}^{n_j}N_{j,i}(t)$ counts the number of 
events in group $j$ up to $t$ and $Y_j(t)=\sum_{i=1}^{n_j}Y_{j,i}(t)$ equals the number of individuals in group $j$ at risk at time $t$. Analogously, the pooled versions $N=N_1+N_2$ and $Y=Y_1+Y_2$ can be interpreted. Using these processes we can introduce the famous Kaplan--Meier and Nelson--Aalen estimators. \citet{ABGK} proved that both estimators obey a central limit theorem, or, in other words, they are asymptotically normal. The Nelson--Aalen estimator $\widehat A_j$ given by
\begin{align*}
\widehat A_j(t)=\int_{[0, t]} \frac{\mathbf{1}\{Y_j>0\}}{Y_j}\,\mathrm{ d }N_j\quad (t\geq 0;\ j=1,2)
\end{align*}
is the canonical nonparametric estimator of the (group specific) cumulative hazard function $A_j(t)=-\log(1-F_j(t))=\int_0^t(1-F_j)^{-1}\,\mathrm{ d }F_j$. Similarly, for the pooled sample we introduce $\widehat A(t)=\int_0^t \mathbf{1}\{Y>0\}/Y \,\mathrm{ d }N$ $(t\geq 0)$. In the following, we need the Kaplan--Meier estimator $\widehat F$ (only) for the pooled sample. It is 
\begin{align*}
1-\widehat F(t)=\prod_{(j,i):X_{j,i}\leq t}\Bigl( 1- \frac{\delta_{j,i}}{Y(X_{j,i})} \Bigr)=\prod_{(j,i):X_{j,i}\leq t}\Bigl( 1- \frac{\Delta N(X_{j,i})}{Y(X_{j,i})} \Bigr)\quad (t\geq 0),
\end{align*}
where $\Delta f(t)=f(t)-f(t-)$ denotes the jump height in $t$ for $f:\R \to \R$. 

In the subsequent sections we study the two-sample testing problem
\begin{align}\label{eqn:testproblem}
H_{=}: F_1=F_2\text{ versus }K_{\neq} : F_1\neq F_2.
\end{align}
Weighted logrank tests are well known and often applied in practice for this testing problem. An introduction to these tests in their general form can be found in the books of \citet{ABGK} and \citet{FlemingHarrington}. First, choose a weight function $w\in \mathcal W =\{w:[0,1]\to\R\text{ continuous and of bounded variation}\}$.  Then the corresponding weighted logrank statistic is 
\begin{align*}
T_n(w)= \Bigr(\frac{n}{n_1n_2}\Bigl)^{1/2} \int_{[0, \infty)} w(\widehat F(t-)) \frac{Y_1(t)Y_2(t)}{Y(t)}\Bigl[ \,\mathrm{ d }\widehat A_1(t) -\,\mathrm{ d }\widehat A_2(t)\Bigr].
\end{align*}
By \citet{Gill_1980} $T_n(w)$ is asymptotically normal and its asymptotic variance can be estimated by
\begin{align}\label{eqn:sigma_Gill}
\widehat\sigma_n^2 (w) = \frac{n}{n_1n_2} \int_{[0, \infty)} w(\widehat F(t-))^2 \frac{Y_1(t)Y_2(t)}{Y(t)} \,\mathrm{ d }\widehat A(t).
\end{align}
Tests based on $T_n(w)$ or studentized versions based on $T_n(w)/\widehat\sigma_n(w)$ are not omnibus tests for \eqref{eqn:testproblem}. But they have good properties for specific semiparametric hazard alternatives depending on the pre-chosen weight function $w$. Among others, $T_n(w)$ is consistent for alternatives of the form
\begin{align}\label{eqn:alternative}
K_w: A_2(t)=\int_{[0,t]} 1+ \vartheta w\circ F_1 \,\mathrm{ d } A_1, 
\end{align}
where we consider all $\vartheta\neq 0$ leading to a non-negative integrand $1+\vartheta w\circ F_1$ over the whole line. 
For example, the classical logrank test with weight $w\equiv 1$ is consistent against the proportional hazard alternative $K_{\text{prop}}: A_2(t)=(1+\vartheta)A_1(t)$, $0\neq \vartheta\in(-1,\infty)$, and even optimal for so-called local alternatives $K_{\text{loc}}:A_2(t)=(1+n^{-1/2}\vartheta)A_1(t)$, see \citet{Gill_1980}. Choosing $w_{\text{prop}}\equiv 1$ we weight all time points equally. Instead of this, we can also give more weight to departures of the null $A_1=A_2$ at early times by setting $w_{\text{early}}(u)=u(1-u)^3$ or at central times, which are close to the median $F_1(1/2)$, by $w_{\text{cent}}(u)=u(1-u)$ $(u\in[0,1])$. All these are examples for stochastic ordered alternatives, i.e., we have $F_1\leq F_2$ or $F_1\geq F_2$, depending on the sign of $\vartheta$. Even the local increments $A_2(t,t+\varepsilon]$ are ordered since all $w$ are strictly positive. An example without the latter property is the crossing hazard weight $w_{\text{cross}}(u)=1-2u$ with a sign switch at $u=1/2$. Since $w_{\text{prop}}$ and $w_{\text{cross}}$ are orthogonal in $L^2(0,1)$, i.e., $\int_0^1w_{\text{prop}}w_{\text{cross}}(x)\,\mathrm{ d }x=0$, it is not surprising that the classical logrank test has no asymptotic power for the crossing hazard alternative $K_w$ with $w=w_{\text{cross}}$, and vice versa.  Our paper's aim is to combine the good properties of $T_n(w)$ for different weight functions $w$ to obtain a powerful test for various hazard alternatives simultaneously. 

\section{Our test and its asymptotic properties}\label{sec:uncond_test}
For the asymptotic set-up we need two (very common) assumptions. First, assume that no group size vanishes: $0< \liminf_{n\to\infty}{n_1}/{n} \leq \limsup_{n\to\infty}{n_1}/{n}<1$. Let $\tau = \inf\{u>0:[1-G_1(u)][1-G_2(u)][1-F_1(u)][1-F_2(u)]=0\}$, where the convention $\inf \emptyset =\infty$ is used. To observe not only censored data it is convenient to suppose $F_1(\tau)>0$ or $F_2(\tau)>0$ in the case of $\tau<\infty$.

The basic idea of our test is to first choose an arbitrary amount of hazard directions/weights $w_1,\ldots,w_m\in \mathcal W$ $(m\in\N)$ and to consider the vector $T_n=[T_n(w_1),\ldots,T_n(w_m)]^T$ of the corresponding weighted logrank tests. In the spirit of \eqref{eqn:sigma_Gill} let the empirical covariance matrix $\widehat\Sigma_n$ of $T_n$ be given by its entries
\begin{align*}
(\widehat\Sigma_n)_{r,s}=\frac{n}{n_1n_2} \int_{[0,\infty)} w_s(\widehat F(t-))w_r(\widehat F(t-)) \frac{Y_1(t)Y_2(t)}{Y(t)} \,\mathrm{ d }\widehat A(t)\quad (r,s=1,\ldots,m).
\end{align*}
The studentized version of the statistic $T_n$ is the quadratic form $S_n=T_n^T\widehat\Sigma_n^-T_n$, where $A^-$ denotes the Moore--Penrose inverse of the matrix $A$. We suggest to use $S_n$ for testing \eqref{eqn:testproblem}. For our asymptotic results we restrict our considerations to linear independent weights in the following sense:
\begin{assumption}\label{ass:linear_independent}
	Suppose for all $\varepsilon\in(0,1)$ that $w_1,\ldots,w_m$ are linearly independent on $[0,\varepsilon]$, i.e., $\sum_{i=1}^m \beta_iw_i(x) = 0$ for all $x\in[0,\varepsilon]$ implies $\beta_1=\cdots=\beta_m=0$.
\end{assumption}
Many typical hazard weights are polynomial, for example the ones we introduced in Section \ref{sec:intro}. For these weights the linear independence on $[0,1]$ is equivalent to the one on $[0,\varepsilon]$. Consequently, it is easy to check whether the pre-chosen weights fulfill Assumption 1.
\begin{theorem}[Convergence under the null]\label{theo:uncond_null}
	Let Assumption \ref{ass:linear_independent} be fulfilled. Then  $S_n$ converges in distribution to a $\chi_m^2$-distributed random variable.
\end{theorem}
Regarding Theorem \ref{theo:uncond_null} we define our test by $\phi_{n,\alpha}=\mathbf{1}\{S_n>\chi_{m,\alpha}^2\}$ $[\alpha\in(0,1)]$, where $\chi_{m,\alpha}^2$ is the $(1-\alpha)$-quantile of the $\chi_m^2$-distribution. Under Assumption \ref{ass:linear_independent} $\phi_{n,\alpha}$ is asymptotically exact, i.e., $\E_{H_{=}}(\phi_{n,\alpha})\to \alpha$. We want to point out that Assumption \ref{ass:linear_independent} is not needed to obtain distributional convergence under the null, see \citet{BrendelETAL2014}.  But the degree of freedom $k$ of the limiting $\chi_k^2$-distribution depends in general on the unknown asymptotic set-up and may be less than $m$ if Assumption \ref{ass:linear_independent} does not hold. For this case \citet{BrendelETAL2014} suggested to estimate $k$ by its consistent estimator $ \kappa=\text{rank}(\widehat{\Sigma}_n)$ and use the data depended critical value $\widehat c_\alpha=\chi_{\kappa,\alpha}^2$. 

Theorem \ref{theo:uncond_null} implies that the classical statistic $T_n(w)/\widehat\sigma_n(w)$ converges in distribution to a $\chi_1^2$-distributed random variable. The weighted logrank test $\phi_{n,\alpha}(\widetilde w)=\mathbf{1}\{T_n(\widetilde w)/\widehat\sigma_n(\widetilde w)>\chi^2_{1,\alpha}\}$ of asymptotic exact size $\alpha\in(0,1)$ is consistent for alternatives of the shape \eqref{eqn:alternative} with $w\widetilde w\geq 0$ and $\int w(x)\widetilde w(x) \,\mathrm{ d }x>0$. This can be concluded, for instance, from the subsequent Theorem \ref{theo:power_local_altern}. For $\widetilde w=w_i$ this consistency can be transferred to our $\phi_{n,\alpha}$ and, consequently, we combine the strength of each single weighted logrank test.
\begin{theorem}[Consistency]\label{theo:uncond_consis}
	Consider a fixed alternative $K$. Suppose for some $i=1,\ldots,m$ that $\phi_{n,\alpha}(w_i)$ is consistent for testing $H_=$ versus $K$, i.e., the error of second kind $E_K[1-\phi_{n,\alpha}(w_i)]$ tends to $0$ for all $\alpha\in(0,1)$. Then $\phi_{n,\alpha}$ is consistent as well. 
\end{theorem}
Consequently, our test $\phi_{n,\alpha}$ is consistent for alternatives \eqref{eqn:alternative} with $w$ coming from the linear subspace $\mathcal W_m=\{ \sum_{i=1}^{m}\beta_iw_i:\beta=(\beta_1,\ldots,\beta_m)\in\R^m,\;\beta\neq 0\}$ of $\mathcal W$ or, more generally, with $w$ such that $ww_i \geq 0$ and $\int w(x)w_i(x)\,\mathrm{ d }x>0$ for some $i=1,\ldots,m$. Having this in mind the statistician should make his choice for the weights. 

In the introduction we already mentioned local alternatives, which are small perturbations of the null assumption $F_1=F_2$, or equivalently $A_1=A_2$. 
Let $F_0$ be a continuous (baseline) distribution and $A_0$ be the corresponding (baseline) cumulative hazard function. From now on, the survival distributions of both groups depend on the sample size $n$ and we write $F_{j,n}$ as well as $A_{j,n}$ instead of $F_j$ and $A_j$, respectively. Let $A_{1,n}$ and $A_{2,n}$ be perturbations of the baseline $A_0$ in (opposite) hazard directions $w$ and $-w$. To be more specific, let
\begin{align}\label{eqn:local_alternative}
A_{j,n}(t)=\int_{[0,t]} 1+ c_{j,n} w\circ F_0 \,\mathrm{ d } A_0\quad (t\geq 0),\quad c_{j,n}= \frac{(-1)^{j+1}}{n_j}\Bigl( \frac{n_1n_2}{n} \Bigr)^{1/2}
\end{align}
for some $w\in\mathcal W$ and sufficiently large $n$ such that the integrand is nonnegative over the whole line. Clearly, the two regression coefficients $c_{j,n}=O(n^{-1/2})$ are asymptotically of rate $n^{-1/2}$. These coefficients are often used for two-sample rank tests.  We denote by $\E_{n,w}$ the expectation under \eqref{eqn:local_alternative} and by $\E_{n,0}$ the expectation under the null $F_{1,n}=F_{2,n}=F_0$.
\begin{theorem}[Power under local alternatives]\label{theo:power_local_altern}
	Suppose that Assumption \ref{ass:linear_independent} and  $n_1/n\to\eta\in(0,1)$ hold. Define $\psi=[(1-G_1)(1-G_2)]/[\eta(1-G_1)+(1-\eta)(1-G_2)]$. Under \eqref{eqn:local_alternative} $S_n$ converges in distribution to a $\chi_m^2(\lambda)$-distributed random variable with noncentrality parameter $\lambda= a^T\Sigma^-a$, where $a=(\int w\circ F_0 w_i\circ F_0 \psi\,\mathrm{ d } F_0 )_{i\leq m}$ and the entries of $\Sigma$ are $\Sigma_{r,s}=\int w_r\circ F_0w_s\circ F_0\psi\,\mathrm{ d }F_0$ $(1\leq r,s\leq m)$.
\end{theorem}
From the well known properties of noncentral $\chi^2$-distributions we obtain from Theorems \ref{theo:uncond_null} and \ref{theo:power_local_altern} that our test is asymptotically unbiased under local alternatives, i.e., $\E_{n,w}(\phi_{n,\alpha})\to \beta_{w,\alpha} \geq\alpha$. In the proof of Theorem \ref{theo:power_local_altern} we show that the limiting covariance $\Sigma$ is invertible. That is why $a\neq 0$ implies $\lambda=a^T\Sigma^- a>0$ and $\E_{n,w}(\phi_{n,\alpha})\to \beta_{w,\alpha} >\alpha$. Clearly, $w\in\mathcal W_m$ lead to $a\neq 0$ and, hence, our test has nontrivial power for local alternatives in hazard direction $w$ coming from the linear subspace $\mathcal W_m$. For this kind of alternatives the test is even admissible, a certain kind of optimality which says that there is no test which achieves better asymptotic power for all hazard alternatives $w\in\mathcal W_m$ simultaneously. 
\begin{theorem}[Admissibility]\label{theo:uncon_admissible}
	Suppose that Assumption \ref{ass:linear_independent} holds. Then there is no test sequence $\varphi_n$ ($n\in\N$) of asymptotic size $\alpha$, i.e., $\limsup_{n\to\infty}E_{H_=}( \varphi_n)\leq \alpha$, such that $\liminf_{n\to\infty}[\E_{n,w}(\varphi_n)-\E_{n,w}(\phi_{n,\alpha})] $ is nonnegative for all $w\in\mathcal W_m$ and positive for at least one $w\in \mathcal W_m$. 
\end{theorem}

All proofs are deferred to the appendix.

\section{Permutation test}\label{sec:permutation}

Denote by $X_{(1)}\leq \cdots\leq X_{(n)}$ the order statistics of the pooled sample. Let $c_{(k)}\in\{c_{1,n},c_{2,n}\}$ and $\delta_{(k)}\in\{0,1\}$ $(1\leq k \leq n)$ be the group and the censoring status corresponding to $X_{(k)}$, i.e., if $X_{(k)}=X_{j,i}$ then $c_{(k)}=c_{j,n}$ and $\delta_{(k)}=\delta_{j,i}$. The counting processes $N_{j},Y_{j}$ used for the test statistic jump only at the order statistics. Their value at these points can be expressed by the components of $c^{(n)}=(c_{(1)},\ldots,c_{(n)})$ and  $\delta^{(n)}=(\delta_{(1)},\ldots,\delta_{(n)})$, for example $N_j(X_{(k)})=\sum_{i=1}^k \delta_{(i)} \mathbf{1}\{c_{(i)}=c_{j,n}\}$. Consequently, the test statistic only depends on $(c^{(n)},\delta^{(n)})$. That is why we write $S_n(c^{(n)},\delta^{(n)})$ instead of $S_n$ throughout this section. The basic idea of our permutation test is to keep $\delta^{(n)}$ fixed and to permute $c^{(n)}$ only, i.e., to randomly permute the group membership of the individuals. For the case $m=1$, i.e., $S_n=T_n(w)^2/\widehat \sigma_n^2(w)$, this permutation idea was already used by \citet{Neuhaus1993} and \citet{JanssenMayer2001}. In simulations of \citet{Neuhaus1993} and \citet{HellerVenkatraman_1996} the resulting permutation test had a good finite sample performance, even in the case of unequal censoring $G_1\neq G_2$. 

Let $c^\pi_n$ be a uniformly distributed permutation of $c^{(n)}$ and be independent of $\delta^{(n)}$. Denote by $c^*_{n,\alpha}(\delta)$ $(\alpha\in(0,1), \ \delta\in\{0,1\}^n)$ the $(1-\alpha)$-quantile of the permutation statistic $S_n(c_n^\pi,\delta)$. Then our permutation test is given by $\phi_{n,\alpha}^*=\mathbf{1}\{S_n(c^{(n)},\delta^{(n)})>c_{n,\alpha}^*(\delta^{(n)})\}$. This test shares all the asymptotic properties of the unconditional test verified in the previous section.
\begin{theorem}\label{theo:perm}
	Suppose that Assumption \ref{ass:linear_independent} is fulfilled and fix $\alpha\in(0,1)$. Then $\phi_{n,\alpha}^*$ is asymptotically exact under $H_=$ and $\phi_{n,\alpha}^*$ is consistent for fixed alternative $K$ whenever $\phi_{n,\alpha}$ is consistent for $K$. Under local alternatives \eqref{eqn:local_alternative} $\phi_{n,\alpha}^*$ and $\phi_{n,\alpha}$ are asymptotically equivalent, i.e., $E_{n,w}(|\phi_{n,\alpha}-\phi_{n,\alpha}^*|)\to 0$, and, hence, they have the same asymptotic power under local alternatives. In particular, $\phi_{n,\alpha}^*$ is asymptotically admissible, compare to Theorem \ref{theo:uncon_admissible}.
\end{theorem}
Since the distribution of $S_n(c_n^\pi,\delta)$ is discrete we may add a randomisation term in the test's definition: $\phi_{n,\alpha}^*= \mathbf{1}\{S_n(c^{(n)},\delta^{(n)}) > c^*_{n,\alpha}(\delta^{(n)}) \} + \gamma_{n,\alpha}^*(\delta^{(n)})\mathbf{1}\{S_n(c^{(n)},\delta^{(n)}) = c^*_{n,\alpha}(\delta^{(n)}) \}$, $\gamma_{n,\alpha}^*(\delta)\in[0,1]$ $(\delta\in\{0,1\}^n)$. Since $c^{(n)}$ and $\delta^{(n)}$  are independent under the restricted null $ H_{\text{res}}: \{F_1=F_2,\ G_1=G_2\}$, see \citet{Neuhaus1993}, the test with additional randomisation term is even finitely exact, i.e., $E_{H_{\text{res}}}(\phi_{n,\alpha}^*)=\alpha$. 

\section{Simulations}\label{sec:simus}

\subsection{Type-I error}
To analyse the behaviour of the proposed test statistic for small sample sizes, we performed a simulation study implementing various situations. All simulations were conducted with the \textsc{R} computing environment, version 3.2.3 \citet{R} using 10,000 simulation and 1,000 permutation runs. 

First, we considered the behaviour of different tests under the null hypothesis $	H_{=}: F_1=F_2$.
Survival times were generated following an exponential $\Exp(1)$ distribution.
Censoring times were simulated to follow the same distribution as the survival times, but with varying parameters to reflect different proportions of censoring: No censoring, equal censoring in both groups, where the parameters were chosen such that on average 15\% of individuals were censored, and unequal censoring distributions reflecting 10\% and 20\% censoring (on average) in the first and second group, respectively. 
Sample sizes were chosen to construct balanced as well as unbalanced designs, namely $(n_1, n_2) =(50, 50), (n_1, n_2)= (30, 70), (n_1, n_2)=(100, 100)$ and $(n_1, n_2)=(150, 50)$.
For all scenarios, we compared the performance of our test with and without permutation based on weights of the form
\begin{align}\label{eqn:weights_wrg_wcross}
w^{(r,g)}(u)	= u^r(1-u)^g\quad (r,g\in\N_0),\quad w_{\text{cross}}(u)=1-2u,
\end{align}
including the famous weights $w^{(0,0)}$ (\textit{proportional hazards}), $w^{(1,1)}$ (\textit{central hazards}) and $w_{\text{cross}}$  (\textit{crossing hazards}). But also mid-early, early, mid-late and late hazards are included in this class of hazard weights. We distinguished between testing based on two or four hazard directions $w_i$, namely proportional and crossing hazards
$$
w_1(u) = 1, \quad w_2(u) = 1-2u
$$
as well as additionally central and early hazards 
$$
w_3(u) = u(1-u),\quad w_4(u) = u(1-u)^3.
$$
The resulting type-I error rates are displayed in Table \ref{tab:typIerror_exp}. As we can see from the tables, the permutation version of the test always keeps the nominal level of 5\% better than the corresponding $\chi^2$-approximation. Testing based on two or four directions, in contrast, does not change the type-I error much for neither the permutation test nor the $\chi^2$-approximation.

\begin{table}[ht]
	\centering
	\caption{Type-I error rates in \% (nominal level 5\%) for exponentially distributed censoring and survival times, testing based on two or four hazard directions with and without permutation procedure, respectively}
	\label{tab:typIerror_exp}
	\begin{tabular}{cccccc}
		\multicolumn{2}{c}{}& \multicolumn{2}{c}{Permutation test} & \multicolumn{2}{c}{$\chi^2$-Approximation}\\
		$(n_1, n_2)$ & censoring & { 4 directions} & { 2 directions} & { 4 directions} & { 2 directions}\\ [5pt]
		
		\multirow{3}{*}{(50, 50)} & None & 5$\cdot$12 & 5$\cdot$35 & 6$\cdot$30 & 6$\cdot$09\\ 
		& equal & 4$\cdot$90 & 5$\cdot$13 & 5$\cdot$74 & 5$\cdot$83 \\ 
		& unequal & 4$\cdot$85 & 4$\cdot$54 & 5$\cdot$79 & 5$\cdot$19 \\ 
		\multirow{3}{*}{(30, 70)} & none & 4$\cdot$67 & 4$\cdot$91 & 6$\cdot$61 & 5$\cdot$93\\ 
		& equal & 5$\cdot$14 & 4$\cdot$88 & 6$\cdot$74 & 5$\cdot$86 \\ 
		& unequal & 4$\cdot$97 & 5$\cdot$22 & 6$\cdot$75 & 6$\cdot$13\\
		\multirow{3}{*}{(100, 100)} & none & 4$\cdot$71 & 5$\cdot$28 & 5$\cdot$58 & 5$\cdot$55 \\ 
		& equal & 5$\cdot$27 & 5$\cdot$03 & 5$\cdot$99 & 5$\cdot$28 \\ 
		& unequal & 4$\cdot$93 & 5$\cdot$26 & 5$\cdot$61 & 5$\cdot$56 \\ 
		\multirow{3}{*}{(150, 50)} & none & 4$\cdot$77 & 5$\cdot$05 & 5$\cdot$97 & 5$\cdot$70 \\ 
		& equal & 4$\cdot$96 & 5$\cdot$03 & 6$\cdot$25 & 5$\cdot$61 \\ 
		& unequal & 5$\cdot$52 & 5$\cdot$38 & 6$\cdot$79 & 6$\cdot$08 \\ 
	\end{tabular}
\end{table}

%

\subsection{Power behaviour against various alternatives}\label{sec:simu_alternative}

In a second simulation study, we considered the power behaviour of the test under various alternatives using 1,000 simulation and 1,000 permutation runs. Since we found the $\chi^2$-approximation to be slightly liberal in all considered scenarios, we excluded it from the power comparisons.
We again considered the exponential distribution, i.e., survival times in the first group were simulated to follow an $\Exp(1)$ distribution. The simulated data for the second group was generated  according to
$$
A_2(t)=\int_{[0,t]} 1+ \vartheta w_i\circ F_1 \,\mathrm{ d } A_1 
$$
with different weight functions $w_i$ $(i=1, \dots, 4)$ as above. 
Realizations of the distribution belonging to $A_2$ were generated using an acceptance-rejection procedure. The parameter $\vartheta$ was chosen to range from $\vartheta=0$ (corresponding to the null hypothesis) to $\vartheta=$ 0$\cdot$9 in the case of proportional and crossing hazards, to $\vartheta=$ 4$\cdot$5 for central hazards and early hazards.
Censoring times were simulated as above to create equal as well as unequal censoring distributions. Sample sizes 
were $(n_1, n_2)=(50, 50)$ and $(n_1, n_2)= (30, 70)$.
For each alternative based on a weight function $w_i$, we considered our permutation test based on the two or four hazard directions $w_i$ stated above as well as the optimal test based on $T_n(w_i)/\widehat{\sigma}_n(w_i)$ and one of the other one-directional tests based on $T_n(w_j)/\widehat{\sigma}_n(w_j)$ for some $j\neq i$.
In the scenario with early hazards below (Figure \ref{fig:early}), we considered a more extreme choice of early hazard alternatives corresponding to $\tilde{w}_4(u) = (1-u)^5$.

Figures \ref{fig:proportional}--\ref{fig:early} show that choosing the wrong weight function can lead to a substantial loss in power, as already known in the literature. Moreover, both permutation tests follow the power curve of the optimal test. Furthermore, there is no notable difference between equal and unequal censoring proportions, while unbalanced designs tend to result in slightly lower power than balanced designs.  Since the classical logrank test is consistent for early, central and late hazard alternatives it is not surprising that the two-direction test has reasonable power in all scenarios. In Figure \ref{fig:early} the power line of the four--direction test intersect the one of the two--direction test and is even significantly higher for large $\vartheta$. This is an interesting phenomenon indicating two competing effects. On the one hand, we want to choose the true/best direction, but on the other hand, we should not choose too many weights since we would broaden the power into too many directions. In Figure \ref{fig:early} we see that only for a high weight effect size the benefit of choosing the right direction can compensate the negative effect of choosing too many weights. In all other scenarios, the two--direction test  has higher power than the four--direction test. Due to these observations we advice to use the two--direction test unless specific alternatives are more relevant or interesting for the underlying statistical analysis.

\begin{figure}
	\includegraphics[ width = \textwidth]{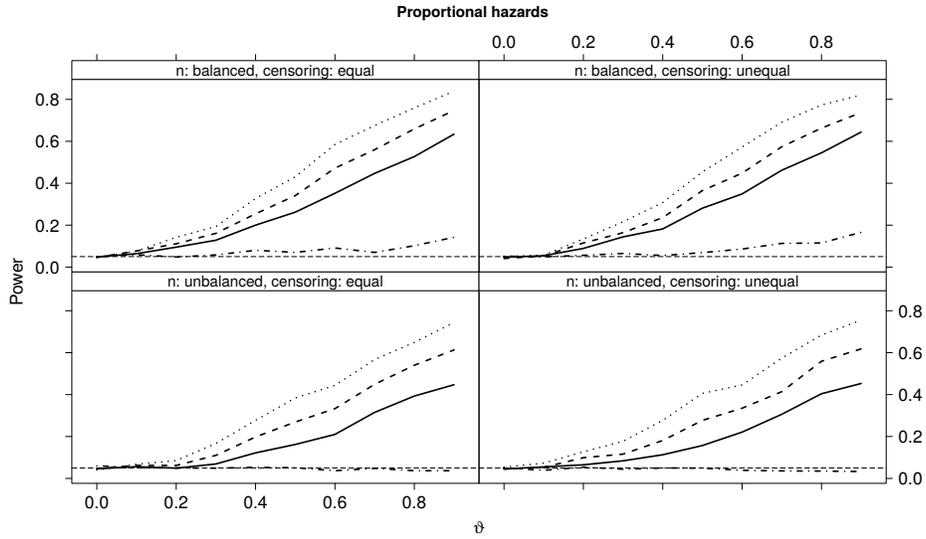}
	\caption{Power simulation results ($\alpha = 5\%$) of the permutation test $\phi^*_{n, \alpha}$ {based on} four (solid) and two (dashed) directions, the proportional hazards (logrank) test (dotted) and the crossing hazards test (dot-dash). Sample sizes are $(n_1,n_2)=(50,50)$ (balanced) and $(n_1,n_2)=(30,70)$ (unbalanced).}
	\label{fig:proportional}
\end{figure}

\begin{figure}
	\includegraphics[ width = \textwidth]{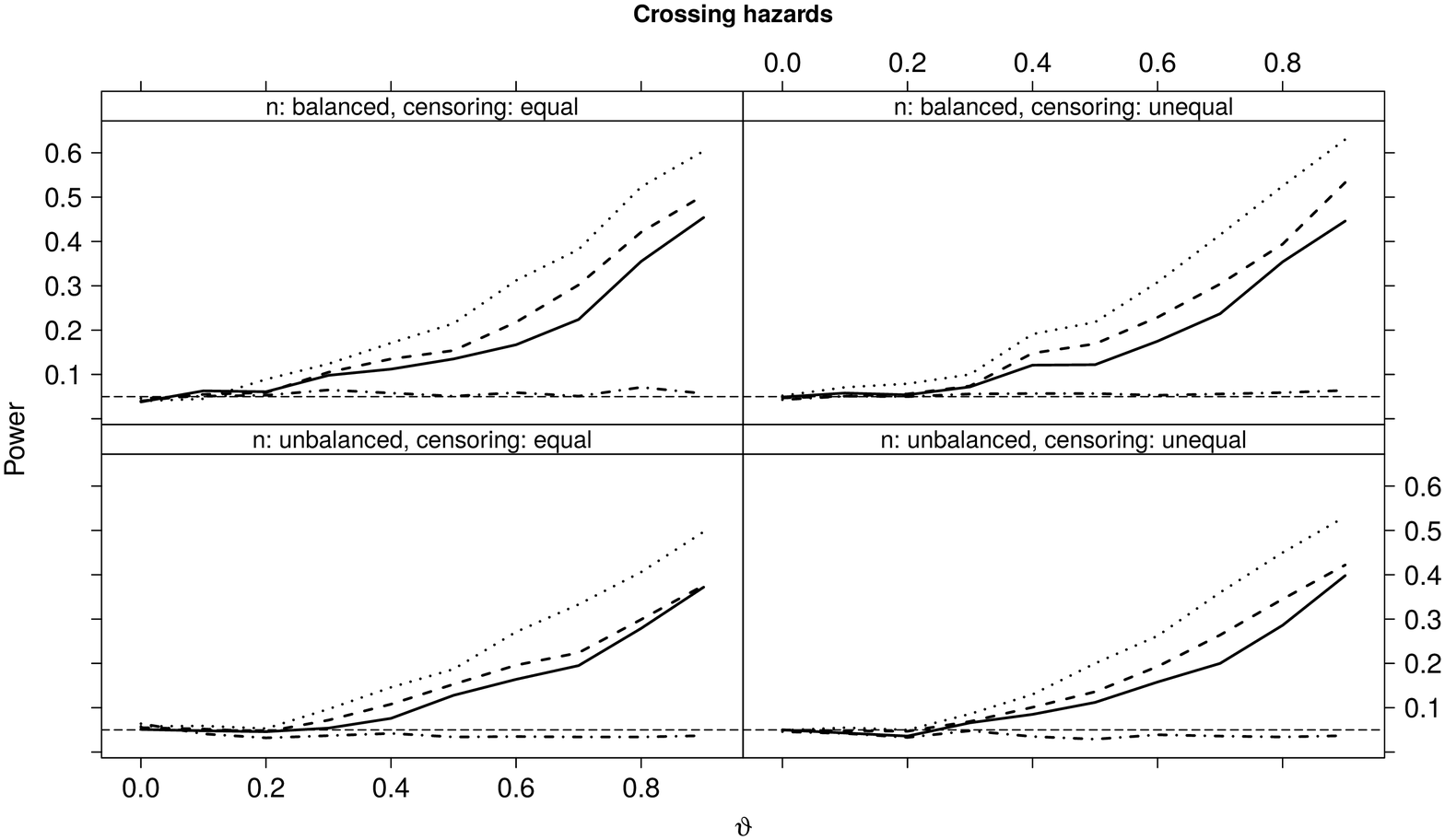}
	\caption{Power simulation results ($\alpha = 5\%$) of the permutation test $\phi^*_{n, \alpha}$ {based on} four (solid) and two (dashed) directions, the crossing hazards test (dotted) and the proportional hazards test (dot-dash). Sample sizes are $(n_1,n_2)=(50,50)$ (balanced) and $(n_1,n_2)=(30,70)$ (unbalanced).}
	\label{fig:crossing}
\end{figure}

\begin{figure}
	\includegraphics[ width = \textwidth]{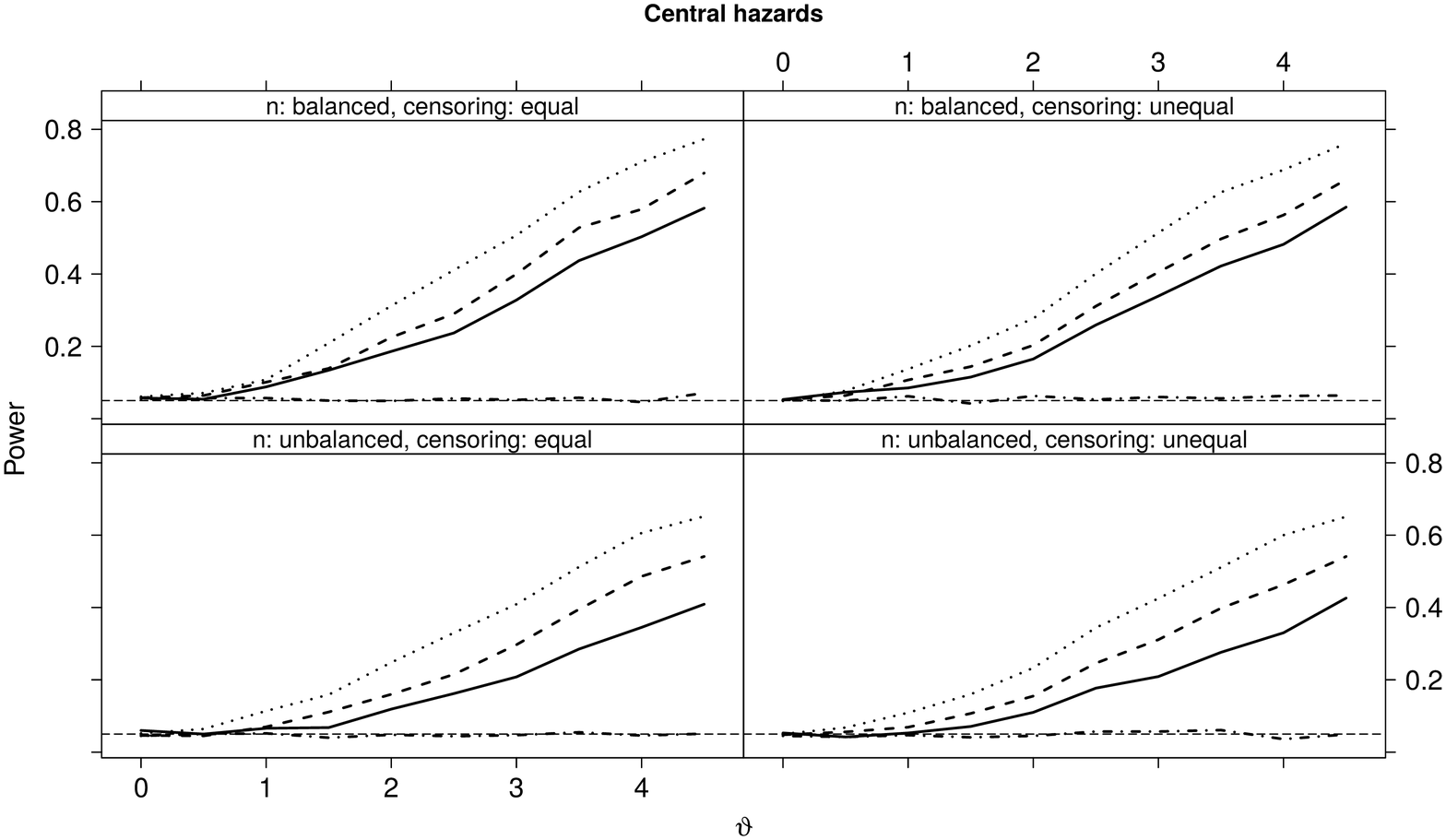}
	\caption{Power simulation results ($\alpha = 5\%$) of the permutation test $\phi^*_{n, \alpha}$ {based on} four (solid) and two (dashed) directions, the central hazards test (dotted) and the crossing hazards test (dot-dash). Sample sizes are $(n_1,n_2)=(50,50)$ (balanced) and $(n_1,n_2)=(30,70)$ (unbalanced).}
	\label{fig:central}
\end{figure}

\begin{figure}
	\includegraphics[ width = \textwidth]{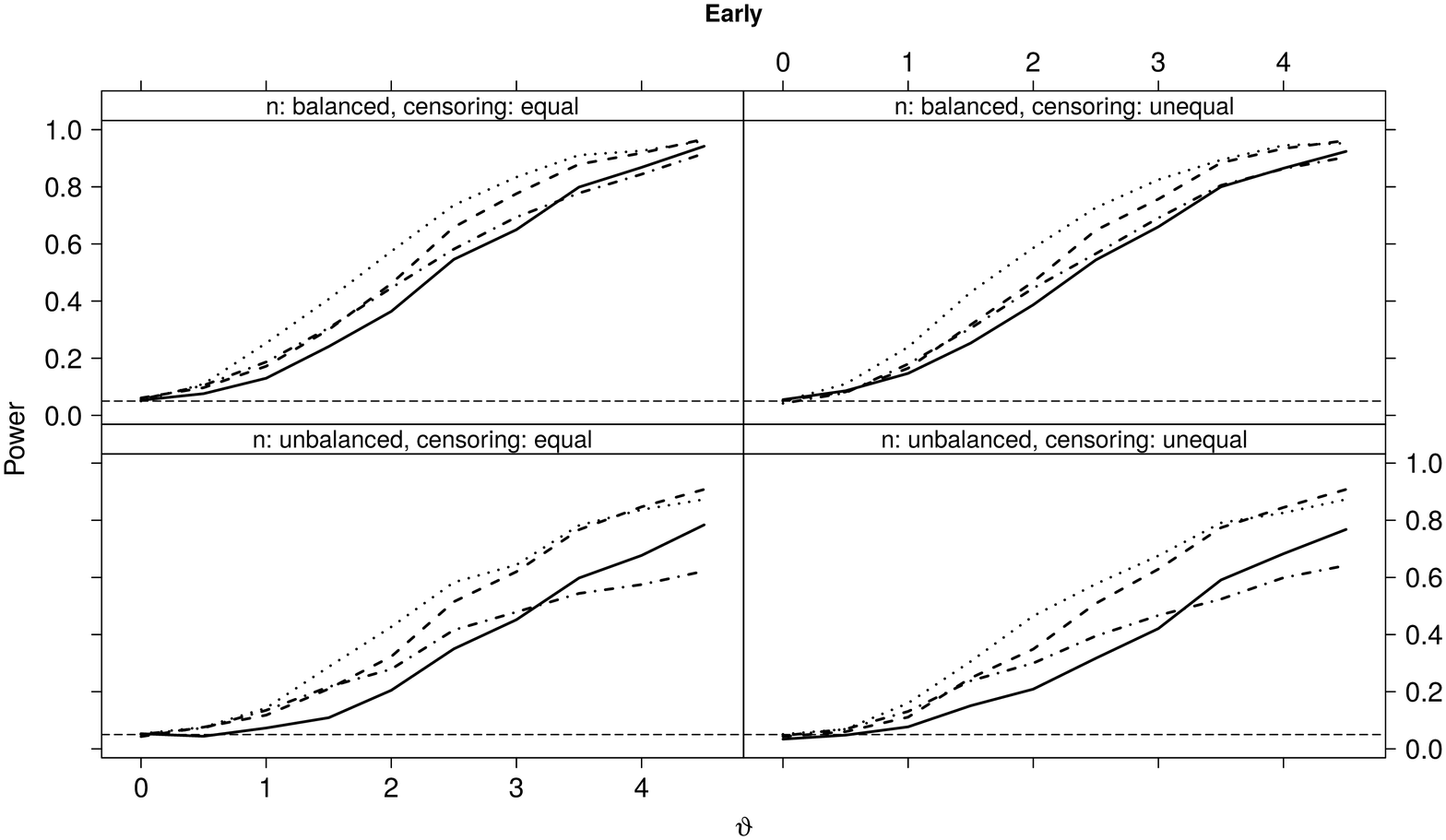}
	\caption{Power simulation results ($\alpha = 5\%$) of the permutation test $\phi^*_{n, \alpha}$ {based on} four (solid) and two (dashed) directions, the early hazards test (dotted) and the crossing hazards test (dot-dash). Sample sizes are $(n_1,n_2)=(50,50)$ (balanced) and $(n_1,n_2)=(30,70)$ (unbalanced).}
	\label{fig:early}
\end{figure}

\section{Real data example}\label{sec:real_data}

As a data example, we reanalyse the gastrointestinal tumor study from \citet{Stablein1981}, which is available in the \emph{coin} package \citet{coin} in \textsc{R}. This study compared the effect of chemotherapy alone versus a combination of radiation and chemotherapy in a treatment of gastrointestinal cancer.	
Of the 90 patients in the study, 45 were randomized to each of the two treatment groups. The Kaplan--Meier curves for the two groups are displayed in Fig.~\ref{fig:km-data}.

\begin{figure}
	\includegraphics[ width = \textwidth]{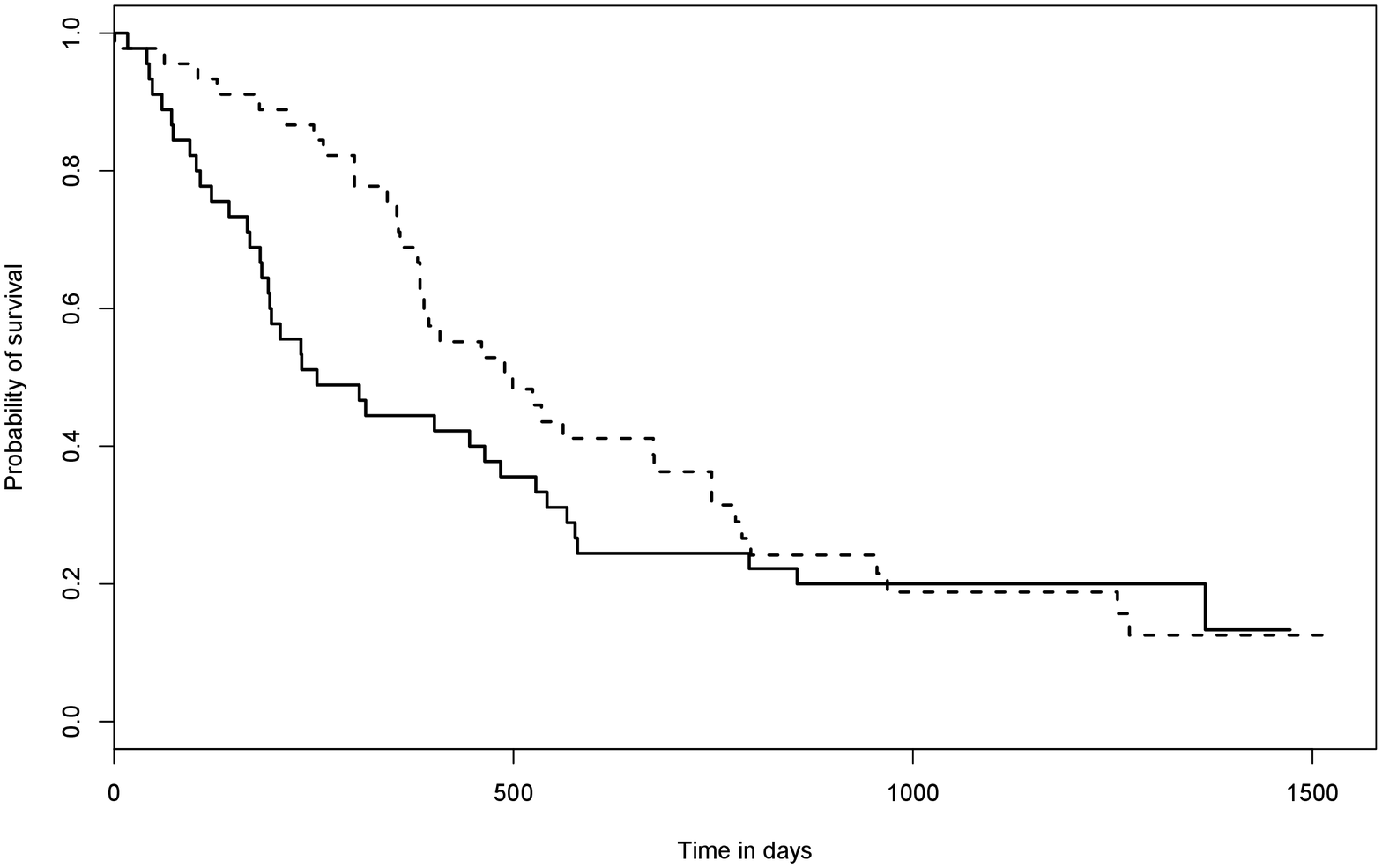}
	\caption{Kaplan--Meier curves for the patients receiving chemotherapy alone (dashed) and those receiving a combination of chemotherapy and radiation (solid).}
	\label{fig:km-data}
\end{figure}

In order to test whether the difference seen between the curves is statistically significant or not, we use our proposed test and its permutation version based on proportional and crossing hazards as well as additionally based on early ($w^{(1, 5)}$) and central hazards. 
After loading the data set in \textsc{R} by \textit{data(GTSG)} the commands \textit{mdir.logrank(GTSG)} and \textit{mdir.logrank(GTSG, cross = TRUE, rg = list(c(0,0), c(1,1), c(1,5)))} do the desired work for the two-direction test (by default) and the four-direction test, respectively. By setting \text{cross=TRUE/FALSE} the user can decide whether the crossing hazard direction $w_{\text{cross}}$ is included and by adding \textit{c(r,g)} to the list \textit{rg} the weight $w^{(r,g)}$, see \eqref{eqn:weights_wrg_wcross}, will be considered in the statistical analysis. By default $10^4$ iterations are used to estimate the permutation quantile. For the users' convenience we implemented a GUI.
We compare the results to the corresponding single-direction tests. The resulting $p$-values are displayed in Table \ref{tab:p-values}. 
As we can see from the table, the single-direction crossing as well as early hazard tests detect significant differences between the two groups at 5\% level, a finding shared by the two- and four-direction tests, while the proportional and the central hazards test do not lead to significant results. This result illustrates the problem when using the classical (single-direction) weighted logrank test since we do not know the right direction in advance. Moreover, the result confirms the advantage of combining different weights and, hence, we advice to use one of our new multiple-direction tests. Similar to the simulation study, we find that the test based on two hazard directions has higher power than the one based on four directions, i.e., the former would still reject the null at 1\% level.

\begin{table}[ht]
	\centering
	\caption{p-values for the single-direction crossing, proportional, early and central hazard tests as well as the multiple-direction tests based on the first two or all four hazard directions in the gastrointestinal cancer study}
	\label{tab:p-values}
	\begin{tabular}{lcccccc}
		&crossing & proportional & early & central & {2 directions} & {4 directions}\\
		permutation&0$\cdot$001  &  0$\cdot$256 &  0$\cdot$005 & 0$\cdot$742 &  0$\cdot$007 & 0$\cdot$017\\ 
		$\chi^2$-approximation &0$\cdot$002  &  0$\cdot$255 &  0$\cdot$005 & 0$\cdot$748 &  0$\cdot$007 & 0$\cdot$018\\
	\end{tabular}
\end{table}

\section{Discussion}\label{sec:discussion}

The main difference between our approach and the one of \citet{BrendelETAL2014} is the additional Assumption \ref{ass:linear_independent}. The linear subset $\mathcal W_m$ of $\mathcal W$ plays an important role, see Theorems \ref{theo:uncond_consis} and \ref{theo:uncon_admissible} as well as the comments to them. Concerning this set it is not an actual restriction to consider only linearly independent weights. As already mentioned, the typical (polynomial) weights fulfill Assumption \ref{ass:linear_independent} if and only if they are linearly independent. Users of our \textsc{R} package \textit{mdir.logrank} do not have to check the linear independence of the weights in advance since we implemented an automatic check. If the pre-chosen weights are linearly dependent then a subset consisting of linearly independent weights will be selected automatically. Consequently, considering additionally Assumption \ref{ass:linear_independent} is not an actual restriction or disadvantage. In fact,  we benefit from this assumption since no additional estimation step for the degree of freedom of the limiting $\chi^2$-distribution under the null is needed. Due to the latter the permutation approach becomes much more computationally efficient. In a similar way the one-sided test of \citet{BrendelETAL2014} for stochastic ordered alternatives $K:\Lambda_1\geq \Lambda_2,\,\Lambda_1\neq \Lambda_2$ may be improved, in particular, concerning computational efficiency. However, due to technical difficulties this is postponed to the future. A further future project is the sample size planning for statistical power of our method.

\section*{Acknowledgement}
The authors thank Markus Pauly for his inspirational suggestions. This work was supported by the \textit{Deutsche Forschungsgemeinschaft}.

\appendix

\section{Proofs}
\subsection{Some notes on \citet{BrendelETAL2014}}
In the subsequent proofs of our theorems we often refer to \citet{BrendelETAL2014}. To avoid a misunderstanding we want to comment on three aspects regarding their results and notation. First, we want to point out that \citet{BrendelETAL2014} interpreted the statistics as certain orthogonal projections. They expressed their test statistic as $||\Pi_{V_r}(\widehat{\gamma}_n)||^2_{\widehat{\mu}_n}$, which equals our $S_n$ according to their Theorem 1. Second, our $w_i$ corresponds to their $\widetilde w_i$ and their $w_i$ in Theorem 9$\cdot$1 equals $w_i\circ F_1$ here. The third aspect concerns the definition of the test statistic. Introduce $m_n=\min\{ \max\{X_{j,i}:i=1,\ldots,n_j\}:j=1,2\}$, the smallest group maximum. Fix $\omega\in \mathcal W$.
\citet{BrendelETAL2014} replaced $w(\widehat F(t-))$ by $w(\widehat F(t-))\mathbf{1}\{t<m_n\}$ $(t\geq 0)$ in the integrands of $T_n(w)$ and $\widehat\Sigma_{r,s}$. Let $T_n^*(w)$ be the corresponding weighted logrank statistic, i.e.,
\begin{align*}
T_n^*(w)= \Bigr(\frac{n}{n_1n_2}\Bigl)^{1/2} \int_{[0,m_n)} w(\widehat F(t-)) \frac{Y_1(t)Y_2(t)}{Y(t)}\Bigl[ \,\mathrm{ d }\widehat A_1(t) -\,\mathrm{ d }\widehat A_2(t)\Bigr].
\end{align*}
All observations lying in $(m_n,\infty)$ belong to the same group, and, hence, the integrand equals $0$ on $(m_n,\infty)$. Consequently, only the set $\{m_n\}$ is excluded from the integration area compared to $T_n(w)$. Since $w$ is bounded we can assume $|w|\leq K\in(0,\infty)$. It is easy to check 
\begin{align*}
|T_n(w)-T_n^*(w)|\leq \Bigr(\frac{n}{n_1n_2}\Bigl)^{1/2} K \Delta N(m_n) \leq K\Bigr(\frac{n}{n_1n_2}\Bigl)^{1/2} \to 0.
\end{align*}
A comparable convergence can be shown for the entries $\widehat\Sigma_{r,s}$ of $\widehat \Sigma$. Finally, the asymptotic results of \citet{BrendelETAL2014} remain valid when we omit the additional indicator  function, as we did in our definitions.

\subsection{Proof of Theorem \ref{theo:uncond_null}}
Considering appropriate subsequences we can assume that $n_1/n\to\eta\in(0,1)$. By Theorem 9$\cdot$1 in the supplement of \citet{BrendelETAL2014} ${T_n}$ converges in distribution to $Z\sim N( 0,\Sigma)$ and $\widehat\Sigma_n$ converges in probability to $\Sigma$, where the entries of $\Sigma$ are 
\begin{align*}
\Sigma_{r,s} =  \int_{[0, \infty)} w_i\circ F_1 w_j \circ F_1\,\psi\;\mathrm{ d }F_1  \quad(1\leq r,s\leq m)
\end{align*}
and $\psi=[(1-G_1)(1-G_2)]/[\eta(1-G_1)+(1-\eta)(1-G_2)]$. Below end we will verify kern$( \Sigma)=\{ 0\}$, i.e., $ \Sigma$ has full rank and is invertible. In this case it is well known that the convergence of the Moore--Penrose inverse follows, i.e., $\widehat\Sigma_n^-\to \Sigma^-$ in probability. By the continuous mapping theorem $S_n$ converges in distribution to a $\chi_m^2$-distributed random variable. Observe that this convergence does not depend on $\eta$ and the subsequence chosen at the proof's beginning.

Let $\beta=(\beta_1,\ldots,\beta_m)^T\in\text{kern}(\Sigma)$. Then
\begin{align*}\label{eqn:conv_null_0=beta...}
0= \beta^T \Sigma \beta = \int_{[0, \infty)} \Bigl( \sum_{i=1}^{m}\beta_i w_i \circ F_1 \Bigr)^2 \psi \,\mathrm{ d }F_1.
\end{align*}
Since $\psi$ is positive on $[0,\tau)$ and $F_1$ as well as $w_1,\ldots,w_m$ are continuous functions we can deduce  $ \sum_{i=1}^{m}\beta_i w_i(x)=0$ for all $x\in[0,F_1(\tau))$. From Assumption \ref{ass:linear_independent}  $\beta_1=\ldots\beta_m=0$ follows.

\subsection{Proof of Theorem \ref{theo:uncond_consis}}
\citet{BrendelETAL2014} showed, see the proof of their Theorem 2, that $S_n \geq T_n(w_i)/\widehat\sigma_n(w_i)$ for all $1\leq i \leq m$. Since consistency of $\phi_{n,\alpha}(w_i)$ implies $\P( T_n(w_i)/\widehat\sigma_n(w_i)> \chi_{1,\alpha}^2)\to 1$ under $K$ for all $\alpha\in(0,1)$ we can deduce that $S_n$ convergences in probability to $\infty$ under the alternative $K$. Finally, the consistency of $\phi_{n,\alpha}$ follows.

\subsection{Proof of Theorem \ref{theo:power_local_altern}}
Following the argumentation of \citet{BrendelETAL2014} for the proof of their Theorem 9$\cdot$1 in the supplement we obtain from Theorem 7$\cdot$4$\cdot$1 of \citet{FlemingHarrington} and the Cram\'{e}r--Wold device that $T_n$ converges in distribution to a multivariate normal distributed $Z\sim N(a,\Sigma)$ and $\widehat{\Sigma}_n \to \Sigma$ in probability. The covariance matrix $\Sigma$ coincides with the one introduced in the proof of Theorem \ref{theo:uncond_null} when replacing $F_1$ by $F_0$. In particular, $\Sigma$ is invertible and (strict) positive definite. By the continuous mapping theorem $S_n$ converges in distribution to a $\chi_m^2(\lambda)$-distributed random variable with noncentrality parameter $\lambda=a^T\Sigma^-a$. 
%
%


\subsection{Proof of Theorem \ref{theo:uncon_admissible}} 
Considering appropriate subsequences we can suppose that $n_1/n\to\eta\in(0,1)$.
Let $Q_{n,\beta}$ $(\beta=(\beta_1,\ldots,\beta_m)^T\in\R^m;\,n\in\N)$ be the common distribution of $(X_{1,1},\delta_{1,1},\ldots,X_{2,n_2},\delta_{2,n_2})$ under the local alternative \eqref{eqn:local_alternative} in direction $w=\sum_{i=1}^m\beta_iw_i$. In particular, $Q_{n,0}$ denotes the corresponding distribution under the null. Let $\psi$ and $\Sigma$ be defined as in Theorem \ref{theo:power_local_altern}.

\begin{lemma}\label{lem:LAN}
	For every $\beta\in\R^m$ the log likelihood ratio can be expressed by
	\begin{align*}
	\log \frac{\mathrm{ d }Q_{n,\beta}}{\mathrm{ d }Q_{n,0}}=\beta^TT_n  - \frac{1}{2}\beta^T\Sigma \beta + R_n,
	\end{align*}
	where $R_n$ converges in $Q_{n,0}$-probability to $0$.
\end{lemma} 
\begin{proof}
	Fix $\beta=(\beta_1,\ldots,\beta_m)^T\in\R^m$ and let $w=\sum_{i=1}^m\beta_iw_i$. Let $\{P_{\theta}^*:\theta\in\Theta\}$, $\Theta=(-\theta_0,\theta_0)\subset \R$, be a parametrized family  with cumulative hazard measures $A_{\theta}^*$ given by
	\begin{align*}
	A_\theta^*(t)=\int_{[0,t]} 1+ \theta w\circ F_0 \,\mathrm{ d } A_0\quad (\theta\in\Theta,\ t\geq 0),
	\end{align*}
	where $\theta_0>0$ is chosen such that the integrand is always positive. Plugging in $\theta=c_{j,n}$ gives us $A_{j,n}$ from \eqref{eqn:local_alternative} $(j=1,2)$. Let $Q_{\theta,j}^*$ ($j=1,2$; $\theta\in\Theta$) be the distribution of $(\min(T,C),\mathbf{1}\{T\leq C\})$ for independent $T\sim P_\theta^*$ and $C\sim G_j$. Obviously, $Q_{n,\beta}= (Q_{c_{1,n},1}^*)^{n_1}\otimes (Q_{c_{2,n},2}^*)^{n_2}$. As already stated by \citet{BrendelETAL2014}, see the top of their page 6, the family $\theta\mapsto Q_{\theta,j}^*$ is $L_2$-differentiable with derivative $L$, say.  Let $M_j=N_j-\int Y_j \,\mathrm{ d } A_0$ $(j=1,2)$. Following the argumentation of \citet{Neuhaus2000}, see also \citet{Janssen1989}, we obtain 
	\begin{align*}
	\log \frac{\mathrm{ d }Q_{n,\beta}}{\mathrm{ d }Q_{n,0}}= Z_n - \frac{1}{2}\sigma^2  + R_n^*,\quad Z_n=\int R(L) \Bigl( \frac{ \mathrm{ d M_1}}{n_1}-\frac{ \mathrm{ d M_2}}{n_2} \Bigr),
	\end{align*}
	where $R_n^ *$ converges in $Q_{n,0}$-probability to $0$, $Z_n$ converges in distribution to $Z\sim N(0,\sigma^2)$ for some  $\sigma\geq 0$ under $Q_{n,0}$ and $R$ is the operator studied by \citet{EfronJohnsrone1990} and \citet{RitovWellner1988}. In our situation $R(L)=w\circ F_0$. It is easy to check that $T_n(w)$ coincides with $Z_{n}$ when we replace $w\circ F_0$ and $A_0$ by $t\mapsto w\circ \widehat F(t-)$ and $\widehat A$, respectively. Using the standard counting process techniques, for example Theorem 4$\cdot$2$\cdot$1 of \citet{Gill_1980}, we can conclude that $T_n(w)-Z_n$ converges in $Q_{n,0}$-probability to $0$. Hence,
	\begin{align*}
	\log \frac{\mathrm{ d }Q_{n,\beta}}{\mathrm{ d }Q_{n,0}}= T_n(w) - \frac{1}{2}\sigma^2  + R_n,
	\end{align*}
	where $R_n$ tends in $Q_{n,0}$-probability to $0$ and $T_n(w)$ converges in distribution to $Z\sim N(0,\sigma^2)$ under $Q_{n,0}$. From the proof of Theorem \ref{theo:uncond_null}, setting $m=1$ and $w_1=w$ there, we get $\sigma^2=\int (w\circ F_0)^2\psi \,\mathrm{ d } F_0$. Finally, observe that $T_n(w)=\beta^TT_n$ and $\sigma^2=\beta^T\Sigma\beta$.
\end{proof}
Recall from the proof of Theorem \ref{theo:power_local_altern} that $T_n$ converges in distribution to $Z\sim N(\Sigma\beta,\Sigma)=Q_\beta$ under $Q_{n,\beta}$ for all $\beta\in\R^m$ and that $\Sigma$ is invertible. Combining these and Lemma \ref{lem:LAN} yields that $\mathrm{ d }Q_{n,\beta}/\mathrm{ d }Q_{n,0}$ converges in distribution under $Q_{n,0}$ to $\mathrm{ d }Q_{\beta}/\mathrm{ d }Q_{0}(Z)$ with $Z\sim Q_0$. In terms of statistical experiments, see Sections 60 and 80 of \citet{Strasser1985}, the experiment sequence $\{Q_{n,\beta}:\beta\in\R^m\}$ fulfills Le Cam's local asymptotic normality, in short LAN, and converges weakly to the Gaussian shift model $\{Q_\beta:\beta\in\R^m\}$. 
\begin{remark}\label{rem:contiguity}
	By Le Cam's first lemma, see Theorem 61$\cdot$3 of \citet{Strasser1985}, $Q_{n,\beta}$ and $Q_{n,0}$ are mutually contiguous for all $\beta\in\R^m$, i.e., convergence in $Q_{n,0}$-probability implies convergence in $Q_{n,\beta}$-probability, and vice versa. 
\end{remark}

From the distributional convergence of $T_n$ mentioned above, we obtain for all $\beta\in\R_m$
\begin{align*}
\E_{n,\beta}(\phi_{n,\alpha})=\int \mathbf{1}\{ x^T\Sigma_n^-x>\chi_{m,\alpha}^2\} \,\mathrm{ d }Q_{n,\beta}^{T_n}(x)\to \int \mathbf{1}\{ x^T\Sigma^-x>\chi_{m,\alpha}^2\} \,\mathrm{ d }Q_\beta(x),
\end{align*}
where $Q_{n,\beta}^{T_n}$ is the image measure of $Q_{n,\beta}$ under the map $T_n$. Since $x\mapsto x^T\Sigma^-x$ is convex we can deduce from Stein's Theorem, see Theorem 5$\cdot$6$\cdot$5 of \citet{Anderson2003}, that $x\mapsto \phi_\alpha^*(x)=\mathbf{1}\{ x^T\Sigma^-x>\chi_{m,\alpha}^2\}$ $(x\in\R^m)$ is an admissible test in the Gaussian shift model $\{Q_{\beta}:\beta\in\R^m\}$ for testing the null $H:\beta=0$ versus the alternative $K:\beta\neq 0$. This means that there is no test $ \varphi$ of size $\alpha$ such that $\int \varphi - \phi^*_\alpha \,\mathrm{ d }Q_\beta$ is nonnegative for all $\beta\neq 0$ and positive for at least one $\beta$. Now, suppose contrary to the claim of Theorem \ref{theo:uncon_admissible} that there is a test sequence $\varphi_n$ ($n\in\N$) with the mentioned properties. By Theorem 62$\cdot$3 of \citet{Strasser1985}, which goes back to Le Cam, there is a test $\varphi$  for the limiting model $\{Q_{\beta}:\beta\in\R^m\}$ such that along an appropriate subsequence $\E_{n,\beta}(\varphi_n)\to \int \varphi \,\mathrm{ d }Q_\beta$ for all $\beta\in\R^m$. Under our contradiction assumption we obtain $ \int \varphi \,\mathrm{ d }Q_0\leq \alpha$, $\int \varphi \,\mathrm{ d } Q_\beta\geq \int \phi^*_\alpha\,\mathrm{ d }Q_\beta$ for all $0\neq\beta\in\R^m$ and $\int \varphi \,\mathrm{ d } Q_\beta> \int \phi^*_\alpha\,\mathrm{ d } Q_\beta$ for at least one $\beta\neq 0$. But, clearly, this contradicts the admissibility of $\phi^*_\alpha$.

\subsection{Proof of Theorem \ref{theo:perm}}
As we explain at the proof's end all statements follow from the subsequent lemma.
\begin{lemma}\label{lem:perm_asym_equi}
	Let $F_1,F_2,G_1,G_2$ be fixed and independent of $n$. Suppose that $S_n(c^{(n)},\delta^{(n)})$ converges in distribution to a random variable $Z$ on $[0,\infty]$. Moreover, assume that the distribution function $t\mapsto \P(Z\leq t)$ $(t\in[0,\infty])$ of $Z$ is continuous on $[0,\infty)$. Then the unconditional test $\phi_{n,\alpha}$ and the permutation test $\phi_{n,\alpha}^*$ are asymptotically equivalent, i.e., $\E(\vert \phi_{n,\alpha}-\phi_{n,\alpha}^*\vert)\to 0$ for all $\alpha\in(0,1)$.
\end{lemma} 
\begin{proof}
	Considering appropriate subsequences we can suppose $n_1/n\to\eta\in(0,1)$. By Lemma 1 of \citet{JanssenPauls2003}  it is sufficient to verify 
	\begin{align*}
	\sup_{x\geq 0} \bigl|\ \P( S_n(c^{(n)}, \delta^{(n)})\leq x \mid \delta^{(n)} ) - \chi^2_m([0,x])\ \bigr| \to 0 
	\end{align*} 
	in probability. Recall that $\xi_n\to \xi$ in probability if and only if every subsequence has a subsequence such that along this sub-subsequence $\xi_n$ converges to $\xi$ with probability one. Define  
	\begin{align*}
	& H(x)=1-\eta[1-F_1(x)][1-G_1(x)]-[1-\eta][1-F_2(x)][1-G_2(x)]\quad(x\geq 0),\\
	& H^1(x)=\eta\int_{[0, x]}(1-G_1)\,\mathrm{ d }F_1+(1-\eta)\int_{[0,x]}(1-G_2)\,\mathrm{ d }F_2\quad(x\geq 0), \\
	& F^*(x)=1-\exp\Bigl( -\eta \int_{[0,x]} \frac{1-G_1}{1-H} \,\mathrm{ d }F_1 -(1-\eta)\int_{[0,x]} \frac{1-G_2}{1-H} \,\mathrm{ d }F_2\Bigr)\quad(x\geq 0).
	\end{align*}
	Let B be an $m\times m$-matrix with entries ${B}_{r,s}= \int w_r(F^*) w_s(F^*)\,\mathrm{ d }H^1$ $(1\leq r,s\leq m)$. Following the proof's argumentation of \citet{BrendelETAL2014} for their Theorem 4, in particular using their Lemmas 10$\cdot$1 and 10$\cdot$2, we can deduce: for every subsequence there is a subsequence such that along this sub-subsequence $S_n(c_n^\pi,\delta^{(n)}(\omega))$ converges in distribution to $\widetilde Z\sim \chi^2_{\text{rank}{(B)}}$ for almost all $\omega$, i.e., for all $\omega\in E$ with $\P(E)=1$.  Consequently, it remains to show $\text{rank}(B)=m$, or equivalently $\text{kern}(B)=\{0\}$. 
	
	Let $\beta=(\beta_1,\ldots,\beta_m)^T\in\text{kern}(B)$. First, observe that for every $0<x<y$ we have $F^*(y)-F^*(x)>0$ if and only if $H^1(y)-H^1(x)>0$. Thus, we obtain from 
	$0 = \beta^TB\beta=\int ( \sum_{i=1}^{m}\beta_iw_i(F^*) )^2\,\mathrm{ d }H^1$ and the continuity of $F^*$ as well as of $w_1,\ldots,w_m$ that $\sum_{i=1}^{m}\beta_iw_i(x)=0$ for all $x\in[0,F^*(\infty)]$, where $F^*(\infty)=\lim_{u\to\infty}F^*(u)$. Since $F_1(\tau)>0$ or $F_2(\tau)>0$ we can conclude $F^*(\infty)\geq F^*(\tau)>0$ and, hence,  $\beta= 0$ follows from the linear independence of $w_1,\ldots,w_m$ on $[0,F^*(\infty)]$.
\end{proof} 
First, suppose that $\phi_{n,\alpha}$ is consistent for a fixed alternative $K$, i.e., $\E(\phi_{n,\alpha})\to 1$ for all $\alpha\in(0,1)$. Then $S_n$ converges to $Z\equiv \infty$ in probability under $K$. Applying Lemma \ref{lem:perm_asym_equi} yields that $\phi_{n,\alpha}^*$ is consistent for $K$ as well.
From Theorem \ref{theo:uncond_null} and Lemma \ref{lem:perm_asym_equi} we can conclude that $\phi_{n,\alpha}^*$ is asymptotically exact. To be more specific, we obtain $E_{n,0}(|\phi_{n,\alpha}-\phi_{n,\alpha}^*|)\to 0$ for all $\alpha\in(0,1)$. From Remark \ref{rem:contiguity}, setting $m=1$ and $w_1=w$ there, we get $E_{n,w}(|\phi_{n,\alpha}-\phi_{n,\alpha}^*|)\to 0$ for all $\alpha\in(0,1)$ and every $w\in\mathcal W$. Combining this and Theorem \ref{theo:uncon_admissible} proves the last statement of Theorem \ref{theo:perm}, the admissibility of $\phi_{n,\alpha}^*$.


\begin{thebibliography}{100}	
	\bibitem[Andersen et~al.(1993)]{ABGK}
	\textsc{Andersen, P.K., Borgan, \O{}, Gill, R.D. \& Keiding, N.} (1993). 
	\newblock \textit{Statistical Models Based on Counting Processes}.	
	\newblock Springer, New York.
	
	\bibitem[Anderson(2003)]{Anderson2003}
	\textsc{Anderson, T.W.} (2003). 
	\newblock \textit{An Introduction to Multivariate Statistical Analysis. Third edition}.	
	\newblock John Wiley \& Sons.
	
	\bibitem[Bagdonavi\v cius et~al.(2010)]{BagonaviciusETAL2010}
	\textsc{Bagdonavi\v cius, V., Kruopis, J. \& Nikulin,
		M. S.} (2010). 
	\newblock \textit{Non-parametric tests for censored data}.	
	\newblock John Wiley \& Sons.
	
	\bibitem[Bajorski(1992)]{Bajorski1992}
	\textsc{Bajorski, P.} (1992). 
	\newblock {Max-type rank tests in the two-sample problem}.	
	\newblock\textit{App. Math.} \textbf{21}, 371--385.
	
	\bibitem[Bathke et~al.(2009)]{BathkeETAL2009}
	\textsc{Bathke, A., Kim, M.-O. \& Zhou, M.} (2009). 
	\newblock {Combined multiple testing by censored empirical likelihood}.	
	\newblock\textit{J. Stat. Plan. Inference} \textbf{139}, 814--827.
	
	\bibitem[Behnen and Neuhaus(1983)]{BehnenNeuhaus1983}
	\textsc{Behnen, K. \& Neuhaus, G.} (1983). 
	\newblock {Galton's test as a linear rank test with estimated scores and its local 	asymptotic efficiency}.	
	\newblock\textit{Ann. Stat.} \textbf{11}, 588--599.
	
	\bibitem[Behnen and Neuhaus(1989)]{BehnenNeuhaus1989}
	\textsc{Behnen, K. \& Neuhaus, G.} (1989). 
	\newblock \textit{Rank tests with estimated scores and their application}.	
	\newblock B.G. Teubner, Stuttgart.
	
	\bibitem[Brendel et~al.(2014)]{BrendelETAL2014}
	\textsc{Brendel, M., Janssen, A., Mayer, C.-D. \& Pauly, M.} (2013). 
	\newblock {Weighted Logrank Permutation Tests for Randomly Right Censored Life Science Data}.	
	\newblock\textit{Scand. J. Stat.} \textbf{41}, 742--761.
	
	\bibitem[Efron and Johnstone(1990)]{EfronJohnsrone1990}
	\textsc{Efron, B. and Johnstone, I.} (1990). 
	\newblock {Fisher's information in terms of the hazard rate}.	
	\newblock\textit{Ann. Stat.} \textbf{18}, 38--62.
	
	\bibitem[Ehm et~al.(1995)]{Ehm1995}
	\textsc{Ehm, W., Mammen, E. \& Mueller, D. W.} (1995). 
	\newblock {Power robustification of approximately linear tests}.	
	\newblock\textit{J. Am. Stat. Assoc.} \textbf{90}, 1025--1033.
	
	\bibitem[Fleming and Harrington(1991)]{FlemingHarrington}
	\textsc{Fleming, T.R. \& Harrington, D.P.} (1991). 
	\newblock \textit{Counting processes and survival analysis}.	
	\newblock Wiley, New York.
	
	\bibitem[Fleming et~al.(1987)]{FlemingETAL1987}
	\textsc{Fleming, T.R., Harrington, D.P. \& O'Sullivan, M.} (1987). 
	\newblock {Supremum versions of the log-rank and generalized Wilcoxon statistics}.	
	\newblock \textit{J. Am. Stat. Assoc.} \textbf{82}, 312--320.
	
	\bibitem[Gar\'es et~al.(2015)]{GaresETAL2015}
	\textsc{Gar\'es, V., Andrieu, S., Dupuy, J.-F. \& and Savy, N.} (2015). 
	\newblock {An omnibus test for several hazard alternatives in prevention randomized controlled clinical trails}.	
	\newblock \textit{Statistics in Medicine} \textbf{34}, 541--557.
	
	
	\bibitem[Gill(1980)]{Gill_1980}
	\textsc{Gill, R.D.} (1980). 
	\newblock \textit{Censoring and stochastic integrals}.	
	\newblock Mathematical Centre Tracts \textbf{124}, Mathematisch Centrum, Amsterdam.
	
	\bibitem[Harrington and Fleming(1996)]{HarringtonFleming1996}
	\textsc{Harrington, D.P. \& Fleming, T.R.} (1996). 
	\newblock Resampling procedures to compare two survival distributions in the presence of right-censored data.	
	\newblock \textit{Biometrics} \textbf{52}, 1204--1213.
	
	\bibitem[Harrington and Fleming(1982)]{HarringtonFleming1982}
	\textsc{Harrington, D.P. \& Fleming, T.R.} (1982). 
	\newblock A class of rank test procedures for censored survival data.	
	\newblock \textit{Biometrika} \textbf{69}, 553--566.
	
	
	\bibitem[Heller and Venkatraman(1996)]{HellerVenkatraman_1996}
	\textsc{Heller, G. \& Venkatraman, E.S.} (1996). 
	\newblock Resampling procedures to compare two survival distributions in the presence of right-censored data.	
	\newblock \textit{Biometrics} \textbf{52}, 1204--1213.
	
	\bibitem[Hothorn et~al.(2006)]{coin}
	\textsc{Hothorn, T., Hornik, K., van de Wiel, M.~A. \& Zeileis, A.} (2006). 
	\newblock A Lego System for Conditional Inference.
	\newblock \textit{The American Statistician} \textbf{60(3)}, 257--263.
	
	\bibitem[Janssen(2000)]{Janssen2000}
	\textsc{Janssen, A.} (2000).
	\newblock {Global power functions of goodness of fit tests}.	
	\newblock \textit{Ann. Stat.} \textbf{28}, 239--253.
	
	
	\bibitem[Janssen and Mayer(2001)]{JanssenMayer2001}
	\textsc{Janssen, A. \& Mayer, C.-D.} (2001).
	\newblock Conditional Studentized survival tests for randomly censored models.	
	\newblock \textit{Scand. J. Stat.} \textbf{28}, 283--293.
	
	\bibitem[Janssen and Pauls(2003)]{JanssenPauls2003}
	\textsc{Janssen, A. \& Pauls, T.} (2003).
	\newblock {How do bootstrap and permutation tests work?}	
	\newblock \textit{Ann. Stat.} \textbf{31}, 768--806.
	
	\bibitem[Jones and Crowley(1989)]{JonesCrowley1989}
	\textsc{Jones, M.P. \& Crowley, J.} (1989).
	\newblock {A general class of nonparametric tests for survival analysis}.	
	\newblock \textit{Biometrics} \textbf{45}, 157--170.
	
	\bibitem[Jones and Crowley(1990)]{JonesCrowley1990}
	\textsc{Jones, M.P. \& Crowley, J.} (1990).
	\newblock {Asymptotic properties of a general class of nonparametric tests for survival analysis}.	
	\newblock \textit{Ann. Stat.} \textbf{18}, 1203--1220.
	
	\bibitem[Klein and Moeschberger(1997)]{KleinMoeschberger1997}
	\textsc{Klein, J.P. \& Moeschberger, M.L.} (1997). 
	\newblock \textit{Survival analysis: techniques for censored and truncated data}.	
	\newblock Springer, New York.
	
	\bibitem[Kosorok and Lin(1999)]{KosorokLin1999}
	\textsc{Kosorok, M.R. \& Lin, C.-Y.} (1999). 
	\newblock {The versatility of function.indexed weighted log-rank statistics}.	
	\newblock \textit{J. Am. Stat. Assoc.} \textbf{94}, 320--332.
	
	\bibitem[Lai and Ying(1991)]{LaiYing1991}
	\textsc{Lai, T.L. \& Ying, Z.} (1991). 
	\newblock {Rank Regression Methods for Left-Truncated and Right-Censored Data}.	
	\newblock \textit{Ann. Stat.} \textbf{19}, 531--556.
	
	\bibitem[Mantel(1966)]{Mantel1966}
	\textsc{Mantel, N.} (1966).
	\newblock {Evaluation of survival data and two new rank order statistics arising in its consideration}.	
	\newblock \textit{Cancer Chemoth. Rep.} \textbf{50}, 163--170.
	
	
	\bibitem[Neuhaus(1993)]{Neuhaus1993}
	\textsc{Neuhaus, G. } (1993).
	\newblock {Conditional rank tests for two-sample problem under random censorship}.	
	\newblock \textit{Ann. Stat.} \textbf{21}, 1760--1779.
	
	\bibitem[Neuhaus(2000)]{Neuhaus2000}
	\textsc{Neuhaus, G.} (2000).
	\newblock {A method of constructing rank tests in survival analysis}.	
	\newblock \textit{J. Stat. Plan. Inference} \textbf{91}, 481--497.
	
	\bibitem[Peto and Peto(1972)]{PetoPeto1972}
	\textsc{Peto, R. \& Peto, J. } (1972).
	\newblock {Asymptotically efficient rank invariant test procedures (with discussion)}.	
	\newblock \textit{J. Roy. Stat. Soc. A} \textbf{135}, 185--206.
	
	\bibitem[R Core Team(2018)]{R}
	\textsc{R Core Team} (2018)
	\newblock{R: A language and environment for statistical computing.}
	\newblock R Foundation for Statistical Computing,
	Vienna, Austria. URL https://www.R-project.org/.
	
	\bibitem[Ritov and Wellner(1988)]{RitovWellner1988}
	\textsc{Ritov, J. and Wellner, J.} (1988). 
	\newblock {Censoring martingales, and the Cox model}.	
	\newblock\textit{Statistical inference from stochastic processes, Contemp. Math.} \textbf{80}, 191--219.
	
	\bibitem[Stablein et~al.(1981)]{Stablein1981}
	\textsc{Stablein, D. M., Carter, W. H., Jr. and Novak, J. W.} (1981).
	\newblock{ Analysis of survival data with nonproportional hazard functions}.
	\newblock \textit{Controlled Clinical Trials} \textbf{2(2)}, 149--159.
	
	\bibitem[Strasser(1985)]{Strasser1985}
	\textsc{Strasser, H.} (1985).
	\newblock \textit{Mathematical Theory of Statistics.}
	\newblock De Gruyter, Berlin/New York.
	
	\bibitem[Tarone(1981)]{Tarone1981}
	\textsc{Tarone, R.E.} (1981). 
	\newblock {On the distribution of the maximum of the logrank statistic and the modified Wilcoxon statistic}.	
	\newblock\textit{Biometrics} \textbf{37}, 79--85.
	
	\bibitem[Tarone and Ware(1977)]{TaroneWare1977}
	\textsc{Tarone, R.E. \& Ware, J.} (1988). 
	\newblock {On distribution-free tests for equality of survival distributions}.	
	\newblock\textit{Biometrika} \textbf{64}, 156--160.
	
	\bibitem[Yang et~al.(2005)]{YangETAL2005}
	\textsc{Yang, S., Hsu, L. \& Zhao, L.} (2005). 
	\newblock {Combining asymptotically normal tests: case studies in comparison of two groups}.	
	\newblock\textit{J. Stat. Plan. Inference} \textbf{133}, 139--158.
	
	\bibitem[Yang and Prentice(2010)]{YangPrentice2010}
	\textsc{Yang, S. \& Prentice, R.} (2010). 
	\newblock {Improved logrank-type tests for survival data using adaptive weights}.	
	\newblock\textit{Biometrics} \textbf{66}, 30--38.
	
\end{thebibliography}
\end{document}